\documentclass[12pt]{amsart}
\usepackage{amsfonts,amssymb,amscd,amsmath,enumerate,url,verbatim}
 \usepackage[dvips]{epsfig}
 \usepackage{amsgen, amstext,amsbsy,amsopn, amsthm, amsfonts,amssymb,amscd,amsmat
 h,euscript,enumerate,url,verbatim,calc,xypic}
\numberwithin{equation}{section}
 \usepackage{latexsym}
 \usepackage{graphics}
 \usepackage{color}

 \newcommand{\m}{\mathfrak{m} }

\newcommand{\Z}{\mathbb{Z}}
\newcommand{\N}{\mathbb{N}}

\theoremstyle{plain}

\newtheorem{thm}{Theorem}[section]
\newtheorem{lem}[thm]{Lemma}
\newtheorem{cor}[thm]{Corollary}
\newtheorem{prop}[thm]{Proposition}
\newtheorem{rem}[thm]{Remark}

\newtheorem{defin-rem}[thm]{Definition and Remark}

\newcommand{\Tor}{\operatorname{Tor}}

\newcommand{\reg}{\operatorname{reg}}

\newcommand{\rate}{\operatorname{rate}}
\newcommand{\rat}{\operatorname{Rate}}

\begin{document}

\title[ ] { rate and syzigies of modules over Veronese subrings}

\begin{abstract}
Let $K$ be a field, $R$ be a standard graded $K$-algebra and $M$ be
a finitely generated graded $R$-module. The rate of $M$,
$\rate_R(M)$, is
 a measure of the growth of the shifts in the minimal graded free resolution of $M$.

 In this paper, we study the rate of Veronese modules of $M$. More precisely, it is shown that
 $\rate_{R^{(c)}}(M)\leq \lceil \max\{
\rate_{R}(M),\rat(R)\}/c\rceil+\max\{0,\lceil
t^{R}_{0}(M)/c\rceil\},$ for all $c\geq
 1$. This
 extends a result of Herzog  et al.
 As a consequence of this, if $M$ is generated in degree zero, then
$\reg_{R^{(c)}}(M)=0$, for all
 $c\geq \max\{\rate_{R}(M), \rat(R)\}$.

Also, for  powers of the homogeneous maximal ideal $\m$ of $R$, it
is shown that $\rate_{R^{(c)}}(\m^{s}(s))\leq \lceil
\rat(R)/c\rceil$, for all $c\geq 1$. In particular case, we give a
simple proof to  a theorem of Backelin.

 \end{abstract}

 \author[R.~ Ahangari Maleki]{Rasoul Ahangari Maleki}
\email{rahangari@ipm.ir, rasoulahangari@gmail.com}






\maketitle

\smallskip
\section*{Introduction}
Let $R$ be a standard graded $K$-algebra with the homogeneous maximal ideal $\m$ and residue field $K$.
There are several invariants attached to a finitely generated graded $R$-module $M$. One is the Castelnuovo-Mumford regularity, which plays an important role
in the study of homological properties of $M$. This invariant can be infinite. Avramove and peeva in \cite{AP} proved that $\reg_R(K)$ is zero or infinite.
The ring $R$ is called  Koszul if $\reg_R(K)=0$. From certain point of views, Koszul algebras behave homologically as polynomial rings. Avramove and Eisenbud in \cite{A-E} showed that if $R$ is Koszul, then the regularity of every finitely generated graded $R$-module is finite.\\
\indent Another important invariant is the rate of  graded modules.
 The notion of rate for algebras introduced by Backelin \cite{B} and it is generalized in \cite{ABH} for graded modules.
 The rate of a finitely generated graded module $M$ over $R$ is defined by
 $$\rate_R(M):=\sup\{t_{i}^{R}(M)/i  :\\ i\geq 1\},$$  where
  $t_{i}^{R}(M):=\max\{ j: \dim_K(\Tor_{i}^{R}(M,K)_{j})\neq 0)\}$.
  This invariant is always finite (see \cite{ABH}). The Backelin rate of
  the algebra $R$ is denoted by $\rat(R)$ and is equal to $\rate_R(\m(1))$, the rate of
  the unique homogenous maximal ideal of $R$ which is shifted by 1.

 By definition,   $\rat(R)\geq 1$ and the equality holds if and only if $R$ is Koszul.
Indeed, the rate of a graded algebra $R$ is an invariant that
measures how far $R$ is from being Koszul.

 Let $c$ be a positive
integer. The $c$-th Veronese subring of the standard graded
$K$-algebra $R= \bigoplus_{i\geq 0}R_i$ is denoted by $R^{(c)}$ and
defined by $R^{(c)}:=\bigoplus _{i\geq 0}R_{ic}$. Backelin
(\cite{B}) used complex arguments about a lattice of ideals, derived
from a presentation of $R$ as a quotient of a free noncommutative
algebra, to prove that the $c$-th Veronese subring $R^{(c)}$ is a
Koszul algebra for all sufficiently large values of $c$. Indeed he
showed that $\rat(R^{(c)})\leq \lceil\rat(R)/c \rceil$, where
$\lceil r\rceil$ denotes  the smallest integer larger than the real
number $r$. Eisenbud, Reeves and Totaro (\cite{ETR}) started their
work from a request by George Kempf for a simpler proof to Backelin
result. In order to do it,
they showed that $R^{(c)}$ admits a quadratic initial ideal for all
sufficiently large values of $c$.

In their paper (\cite{ABH}) Aramova, B$\breve{a}$rc$\breve{a}$nescu
and Herzog showed  that if $M$ is generated in degree zero then,
$$\rate_{R^{(c)}}(M)\leq \max\{ \lceil \rate_{R}(M)/c \rceil,1\},$$
for all $c\geq \max\{1,\rate_S(R)\}$, where $S$ is a polynomial ring
such that $R$ is a homomorphic image of it. Moreover, from their
result
 if $R$ is a polynomial ring, then the inequality holds for all $c\geq1$.

 The purpose of this paper is to extend and improve these results.
Our main result (Theorem \ref{mainthm}) states that for every
finitely generated graded $R$-module $M$,
$$\rate_{R^{(c)}}(M)\leq \lceil \max\{
\rate_{R}(M),\rat(R)\}/c\rceil+\max\{0,\lceil
t^{R}_{0}(M)/c\rceil\},$$
 for all $c\geq1.$ Therefore, if
$c\geq\rat(R)$, then
$$\rate_{R^{(c)}}(M)\leq \max\{ \lceil
\rate_{R}(M)/c \rceil,1\}.$$
 This extends and improves the result
of Aramova et al. Because  $\max\{1,\rate_S(R)\}\geq \rat(R)$. Also,
their statement for polynomial rings holds for Koszul algebras, as
we expect. Indeed,  if $R$ is  Koszul and $M$ is generated in
degrees $0$, then
$$\rate_{R^{(c)}}(M)\leq \max\{ \lceil \rate_{R}(M)/c \rceil,1\},$$
for all $c\geq 1$. In  a special case, when $M= \m^s(s)$, the $s$-th
power of the homogeneous maximal ideal of $R$ shifted by $s$, we
could modify the inequality and we prove that
$$\rate_{R^{(c)}}(\m^s(s))\leq \lceil \rat(R)/c\rceil,
$$ for all $c\geq1.$ As a consequence of this, we get the result of
Backelin.

Throughout this paper, unless otherwise stated, $K$ is a field and
$R= \oplus_{i\in\N_0 }R_i$ denotes a standard graded $K$-algebra,
i.e. $R_0= K$ and $R$ is generated (as a $K$-algebra) by finitely
many elements of degree one. Also, $M= \oplus_{i\in\Z }M_i$ denotes
a finitely generated graded $R$-module.

\section{\bf Notations and Generalities}
\indent

 In this section we prepare some notations and preliminaries which will be used in the paper.

\begin{rem}\label{nota}

 \begin{enumerate}
 \item  For each $d\in\mathbb{Z}$ we denote by
$M(d)$ the graded $R$-module with $M(d)_{p}=M_{d+p}$, for all $p\in\mathbb{Z}.$

Denote by $\mathfrak{m}$ the maximal homogeneous ideal of $R$, that
is $\m= \oplus_{i\in\N}R_i$. Then, we may consider $K$ as a graded
$R$-module via the identification $K=R/\m$.


\item A minimal graded free resolution of $M$ as an $R$-module is a
complex of free $R$-modules
\[\mathbf{F}= \cdots F_{i}\xrightarrow{\partial_{i}}F_{i-1}\rightarrow\cdots\rightarrow F_{1}\xrightarrow{\partial_{1}} F_{0}\rightarrow 0\]
such that $\textsc{H}_{i}(\mathbf{F})$, the $i$-th homology module
of $\mathbf{F}$, is zero for $i>0$, $\textsc{H}_{0}(\mathbf{F})=M$
and $\partial_{i}(F_{i})\subseteq \mathfrak{m}F_{i-1}$ for all $i\in
\mathbb{N}_{0}$. Each $F_{i}$ is isomorphic to a direct sum of
copies of $R(-j)$, for $j\in \mathbb{Z}$. Such a resolution exists
and any two minimal graded free resolutions of $M$ are isomorphic as
complexes of graded $R$-modules. So, for all $j\in\mathbb{Z}$ and
$i\in\mathbb{N}_{0}$  the number of direct summands of $F_{i}$
isomorphic to $R(-j)$  is an invariant of $M$,
 called the $ij$-th graded Betti number of $M$ and denoted by $\beta_{i j}^{R}(M)$.

  Also, by definition,
 the $i$-th Betti number of $M$ as an $R$-module, denoted by  $\beta^{R}_{i}(M)$, is the rank of $F_{i}$.

 By construction, one has $\beta_{i }^{R}(M)=\dim_{K} \Tor^{R}_{i}(M, K)$ and
$\beta_{i j}^{R}(M)=\dim_{K} \Tor^{R}_{i}(M, K)_{j}.$



\item   For every integer $i$ we set
$$t^{R}_i(M):= \max \{j: \beta_{i j}^{R}(M)\neq 0\},$$
if $\beta_{i}^{R}(M)\neq 0$ and $t^{R}_i(M)=-\infty$ otherwise.

\item The Castelnuovo-Mumford regularity
of $M$ is defined by
\[\reg_{R}(M):= \sup\{t^{R}_i(M)-i : i\in\mathbb{N}_0\}.\]
\end{enumerate}
\end{rem}
\begin{defin-rem}

 $M$ is called Koszul if $\reg_R(gr_{\m}(M))=0$. The ring $R$ is
Koszul if the residue field $K$, as an $R$-module, is Koszul.

The  Castelnuovo-Mumford regularity  plays an important role in the
study of homological properties of $M$ and it is clear that
$\reg_R(M) $ can be infinite. Avramov  and Peeva in \cite{AP} proved
that $\reg_R(K)$ is zero or infinite. Also, Avramov and Eisenbud in
\cite{A-E} showed that if $R$ is Koszul, then the regularity of
every  finitely generated graded $R$-module is finite.
\end{defin-rem}

\section{The rate of modules}
 The notion of rate for algebras introduced by Backelin in \cite{B} and  generalized in \cite{ABH} for graded modules.
The rate of a graded algebra is an invariant that measures how far $R$ is from being Koszul.

\begin{defin-rem}\label{ratrem}
\noindent
 \begin{enumerate}
 \item
The Backelin rate of  $R$ is defined as

\[\rat(R):=\sup\{(t_{i}^{R}(K)-1)/i-1 :\\ i\geq 2\},\]
and generalization of this for modules is defined by
\[\rate_R(M):=\sup\{t_{i}^{R}(M)/i  :\\ i\geq 1\}.\]
A comparison with Backelin's rate shows that
$\rat(R)=\rate_{R}(\m(1))$. Note that with the above notations
$\rate_R(R)=-\infty.$

\item
Let $S\rightarrow R$ be a surjective homomorphism of standard graded
$K$-algebras and $M$ be a finitely generated graded $R$-module.
Then, by a modification of  \cite[1.2]{ABH}, one can see that
\begin{equation}\label{ratineq}
\rate_R(M)\leq \max\{
\rate_S(M),\rate_S(R)\}+\max\{0,t_{0}^{S}(M)\}.
\end{equation}

Also, it turns out that the rate of $M$ is finite (see
\cite[1.3]{ABH}).
\end{enumerate}
\end{defin-rem}
\begin{rem}\label{rat-inq}
Consider a minimal presentation of $R$ as a quotient of a polynomial ring, i.e.
$$R\cong S/I$$
where $S=K[X_{1},\cdots,X_{n}]$ is a polynomial ring and $I$ is an
ideal generated by homogeneous elements of degree $>1$. $I$ is
called a defining ideal of $R$. Let $m(I)$ denotes the maximum of
the degrees of a minimal homogeneous generator of $I$.  It follows
from (the graded version of) \cite[2.3.2]{BH}  that
$t^{R}_{2}(K)=m(I)$, thus one has
\[\rat(S/I)\geq m(I)-1.\]

From the above inequality, one can see that $Rate(R)\geq 1$ and the
equality holds if and only if $R$ is Koszul. So that $Rate(R)$ can
be taken as a measure of how much $R$ deviates from being Koszul.
Also, for a module $M$ which is generated in degree zero
 we have $\rate_R(M)\geq 1$ and the
equality holds if and only if $M$ is Koszul, that is $\reg_R(M)=0.$
\end{rem}

\begin{lem}\label{ineq}
Let
$$\cdots\rightarrow L_n\rightarrow L_{n- 1}\rightarrow \cdots\rightarrow L_1\rightarrow L_0\rightarrow L\rightarrow 0$$
be an exact sequence of graded $R$-modules and homogeneous homomorphisms. Then for all $j\in \Z$
$$t_n(L)\leq\max\{t_{n- i}(L_i) ; 0\leq i\leq n\}.$$
\end{lem}

\begin{proof}
We prove the claim by induction on $n$.

In the case $n= 0$, the result follows using the surjection
$$\Tor_{0}^{R}(L_0, K)_j\rightarrow \Tor_{0}^{R}(L, K)_j.$$
 Now, let $n> 0$ and suppose that the result has been proved for smaller values of $n$. Let $K_1$ be the kernel of the homomorphism $L_0\rightarrow L$. Then, using the exact sequence
$$\cdots\rightarrow L_i\rightarrow \cdots\rightarrow L_1\rightarrow K_1\rightarrow 0,$$
and the inductive hypothesis, we have
$$t_{n-1}(K_1)\leq\max\{t_{n- 1-i}(L_{i+1})| 0\leq i\leq n- 1\}.$$
Now, the desired inequality follows  by considering the long exact
sequence obtained by applying $\Tor^R(-, K)$  to the exact sequence
$$0\rightarrow K_1\rightarrow L_0\rightarrow L\rightarrow0.$$
\end{proof}
In the following lemma we compare the rate of the graded $K$-algebra
$R$ and powers of its homogeneous maximal ideal.
\begin{lem}\label{maxi}
For all integers $s>0$ and $i\geq 0$, one has
$$t^{R}_{i}(\m^s(s))\leq t^{R}_{i}(\m(1)).$$  In particular,
$\rate_R(\m^s(s))\leq \rat(R)$.
\end{lem}
\begin{proof}
We prove the claim by induction on $s$. The case $s=1$ is obvious,
so let $s\geq 2$ and consider the exact sequence
$$0\hookrightarrow \m^s\rightarrow \m^{s-1}\rightarrow \m^{s-1}/\m^s\rightarrow0.$$
By applying $\Tor^R(-, K)$ , we get the  exact sequence
$$\Tor_{i+1}^{R}(\m^{s-1}/\m^s, K)_j\rightarrow \Tor_{i}^{R}(\m^s, K)_j\rightarrow
\Tor_{i}^{R}(\m^{s-1}, K)_j,$$ for all $i\geq 0$. This yields the
inequality
\begin{equation}\label{maxin}
t^{R}_{i}(\m^s(s))\leq
\max\{t^{R}_{i+1}(\m^{s-1}/\m^s),t^{R}_{i}(\m^{s-1}) \}.
\end{equation}
Note that $\m^{s-1}/\m^s\simeq K(-s+1)^n$ for some integer $n$, and
that $t^{R}_{i+1}(K)=t^{R}_{i}(\m)$. Now, using the inequality
(\ref{maxin}) and inductive hypothesis, we conclude the assertion.

\end{proof}
\begin{defin-rem}\label{verinf}
Let $c$ and $d$ be integers such that $c>0$ and $0\leq d\leq c-1$.
Assume that $M$ be a finitely generated graded $R$-module.

\begin{enumerate}
 \item
  Define $R^{(c)}:=\bigoplus_{i\in \mathbb{Z}} R_ic$. Then
$R^{(c)}$ is a standard graded $K$-algebra and is a subring of $R$.
We refer to $R^{(c)}$, with this grading, as the $c$-th Veronese
subring of $R$.
 Then the graded $R$-module $M$ can be considered as a finitely generated graded
$R^{(c)}$-module via $R^{(c)}\hookrightarrow R$. \\
 \item
\indent We define $M^{(c,d)}:=\bigoplus_{i\in\mathbb{Z}}M_{ic+d}$,
an $R^{(c)}$-submodule of $M$. This called the $(c,d)$-th Veronese
submodule of $M$. In the case $d=0$, we denote $M^{(c,0)}$ by
$M^{(c)}$. Note that $M$, as a graded $R^{(c)}$-module,
decomposes in to the direct sum $M=\bigoplus_{d=0}^{c}M^{(c,d)}$.\\
\indent It is easy to see that $(-)^{(c,d)}$ is an exact functor
from the category of  graded $R$-modules to the category of graded
$R^{(c)}$-modules.\\
 \item   Let  $x$ be a real number, then we
denote by $\lceil x\rceil$ the smallest integer larger than the $x$.
Note that the $R^{(c)}$-module $R^{(c,d)}$ is generated in degrees
zero and  for any integer $j$, we have
$R(-j)^{(c,d)}=R^{(c,k_d)}(-\lceil (j-d)/{c}\rceil)$ for some $k_d$
with $0\leq k_d\leq c-1$.

Indeed, let $i_{d}$ be the smallest integer such that $i_{d} c\geq
 j-d$, i.e.  $i_d=\lceil (j-d)/{c}\rceil$. Then
 $$R(-j)^{(c,d)}=\bigoplus_{i\in\mathbb{Z}}R_{ic+d-j}=\bigoplus_{i\in\mathbb{Z}}R_{(i-i_d)c+k_d}= R^{(c, k_d)}(-i_d),$$
where $k_d=i_dc+d-j$.

 In particular cases
\begin{enumerate}
 \item  when $d=0$, we have
 $$R(-j)^{(c)}=R^{(c,k)}(-\lceil
j/c\rceil),$$ for some $k$ with $0\leq k\leq c-1$.\\
 \item when $c=1$ and $d=0$, we get
 $$R(-j)=\bigoplus_{r=0}^{c-1}R(-j)^{(c,r)}=\bigoplus_{r=0}^{c-1}R^{(c,k_r)}(-\lceil
(j-r)/c\rceil),$$ for some $k_r$ with $0\leq k_r\leq c-1.$
\end{enumerate}
\end{enumerate}
\end{defin-rem}
In the following proposition we find an upper bound for the degrees
of generators of   syzygies of $R^{(c,d)}$ as an $R^{(c)}$-module in
terms of the degrees of generators of the syzygies of the maximal
ideal of $R$. This proposition will be use in the main theorem of
the paper, too.
\begin{prop}\label{versyz}
Let $c, d$ and $n$ be integers with $0\leq d\leq c-1$ and $n\geq 0$.
Then
$$t_{n}^{R^{(c)}}(R^{(c,d)})\leq \max\big\{\sum_{j=0}^{u} \lceil t^{R}_{\alpha_{j}}(\m(1))/c\rceil
:  0\leq u\leq n,  0\leq \alpha_{j}\leq n,
\Sigma_{j=0}^{u}\alpha_{j}=n \big\}.$$

\end{prop}

\begin{proof}
  Let $c$ and $d$  be   integers with $c>0$ and
$0\leq d\leq c-1$. Consider the graded $R$-module $\m^d(d)$ which is
generated in degree zero. Then, $(\m^d(d))^{(c)}=R^{(c,d)}$. Also,
assume that
$$\mathbf{F}= \cdots
F_{n}\rightarrow F_{n- 1}\rightarrow\cdots\rightarrow
F_{1}\rightarrow F_{0}\rightarrow 0$$ be the minimal graded free
resolution of $\m^d(d)$ as an $R$-module. Then, applying the exact
functor $(-)^{(c)}$ to $\mathbf{F}$ we get an exact complex of
$R^{(c)}$-modules
$$ \mathbf{F}^{(c)}:\cdots\rightarrow F_{n}^{(c)}\rightarrow F_{n-
1}^{(c)}\rightarrow \cdots\rightarrow  F_{0}^{(c)}\rightarrow
R^{(c,d)}\rightarrow 0.$$
 Let $G_i :=F_{i}^{(c)}$  and note that
$G_i =\bigoplus_{j\in\mathbb{Z}} (R(-j)^{(c)})^{\beta_{i
j}^{R}(\m^d(d))}$, for all $i\geq 0$. Then, in view of lemma
\ref{ineq}, we get
\begin{equation}\label{t-ineq}
t^{R^{(c)}}_n(R^{(c,d)})\leq\max\{t^{R^{(c)}}_{n- i}(G_i) ; 0\leq
i\leq n\},
\end{equation}

 for all $n\in \N_0$.

 Now, we prove the claim by induction on $n$.
 Note that $R^{(c,d)}$ and $\m(1)$ are generated in degree aero, as $R^{(c)}$ and $R$-modules,
  respectively.
  Therefore, in the case where $n= 0$ we have
  $$t^{R^{(c)}}_0(R^{(c,d)})= 0= \lceil
t^{R}_0(\m(1))/c\rceil.$$

  For $n=1$, one has
$$t^{R^{(c)}}_1(R^{(c,d)})\leq\max\{t^{R^{(c)}}_{1}(G_0),
t^{R^{(c)}}_{0}(G_1)\}.$$

 Since $G_0$ is a free $R^{(c)}$-module,
$t^{R^{(c)}}_{j}(G_0)=-\infty$ for all $j>0$. Now, using
\ref{verinf}(3)(a), we get
$$t^{R^{(c)}}_1(R^{(c, d)})\leq
t^{R^{(c)}}_{0}(G_1)\leq\max\{t^{R^{(c)}}_{0}(R^{(c,k_j)})+\lceil
t^{R}_1(\m^d(d))/c\rceil :0\leq k_j\leq c-1 \}.$$ Since
$R^{(c,k_j)}$ is generated in degree zero as an $R^{(c)}$-module,
 using Lemma
\ref{maxi},
$$t^{R^{(c)}}_1(R^{(c,d)})\leq \lceil t^{R}_1(\m(1))/c\rceil,$$
as desired.


 \indent  Now, let $n>1$ and
suppose that the result has been proved for smaller values of $n$.
That is
$$t_{i}^{R^{(c)}}(R^{(c,k)})\leq \max\{\sum_{j=1}^{u} \lceil t^{R}_{\alpha_{j}}(\m(1)) /c\rceil : 1\leq u\leq i, \ 1\leq \alpha_{j}\leq i, \   \Sigma_{j=1}^{u}\alpha_{j}=i\}$$
for all $0 \leq i<n$ and  all $0\leq k\leq c-1$.

  Let $0<
i\leq n$. Since for all $j\in \Z$, by \ref{verinf}(3)(a),
$R(-j)^{(c)}= R^{(c,k_j)}(-\lceil j/c\rceil)$ for some $0\leq
k_j\leq c-1$,  we have $$G_i =\bigoplus_{j\in\mathbb{Z}}
R^{(c,k_j)}(-\lceil j/c\rceil))^{\beta_{i j}^{R}(\m^d(d))}.$$
 Hence,
$$t^{R^{(c)}}_{n- i}(G_i) =\max\{t^{R^{(c)}}_{n-
i}(R^{(c,k_j)})+\lceil t^{R}_i(\m^d(d))/c\rceil \}.$$ Applying lemma
\ref{maxi} and using inductive hypothesis, one has
\begin{equation}\label{vreq2}
t^{R^{(c)}}_{n- i}(G_i) \leq \max\big\{ \sum_{j=0}^{u} \lceil
t_{\alpha_{j}}^{R}(\m(1))/c\rceil +\lceil t_i/c\rceil : 0\leq u\leq
n-i, \ 0\leq \alpha_{j}\leq n-i, \ \Sigma_{j=0}^{u}\alpha_{j}=n-i
\big\}.
\end{equation}
Now, by the inequalities (\ref{t-ineq}) and (\ref{vreq2}), we
conclude that
$$t_{n}^{R^{(c)}}(R^{(c,d)})\leq \max\big
\{\sum_{j=0}^{u} \lceil t_{\alpha_{j}}^{R}(\m(1))/c\rceil : 0\leq
u\leq n, \ \Sigma_{j=0}^{u}\alpha_{j}=n, \  0\leq \alpha_{j}\leq
n\big\},$$ as desired.
\end{proof}
 Now, we prove the main result.

\begin{thm}\label{mainthm}
Let $R$ be a standard graded $K$-algebra and $M$ be a finitely
generated  graded $R$-module. Then for all integers $c\geq 1,$
  $$\rate_{R^{(c)}}(M)\leq \lceil \max\{
\rate_{R}(M),\rat(R)\}/c\}\rceil+ \max\{0,\lceil
t^{R}_{0}(M)/c\rceil\}.$$ In particular  for all integers $s,c\geq
1,$
$$\rate_{R^{(c)}}(\m^s(s))\leq \lceil \rat(R)/c\rceil.$$

\end{thm}
\begin{proof}
Let
$$\mathbf{F}= \cdots F_{i}\rightarrow
F_{i-1}\rightarrow\cdots\rightarrow F_{1}\rightarrow
F_{0}\rightarrow 0$$ be the minimal graded free resolution of $M$ as
an $R$-module. Then, $\mathbf{F}$ is, also, an acyclic complex of
$R^{(c)}$-modules. Applying lemma \ref{ineq}, we get
\begin{equation}\label{tineq}
t^{R^{(c)}}_n(M)\leq\max\{t^{R^{(c)}}_{i}(F_j) ; \  0\leq i,j \
\text{and} \ i+j= n\}.
\end{equation}
In view of 2.5(3)(b), we have
\begin{eqnarray*}
F_j&=& \bigoplus_{s\in \Z}R(-s)^{\beta_{js}^{R}(M)}\\
&=& \bigoplus_{s\in \Z}\big(\oplus_{r= 0}^{c- 1}R^{(c,
k_{s})}(-\lceil(s- r)/c\rceil)\big)^{\beta_{js}^{R}(M)}, \ \ for \
some \ 0\leq k_s\leq c- 1.
\end{eqnarray*}

 Therefore,
\begin{eqnarray*}
t^{R^{(c)}}_{i}(F_j)&=&\max\{t_{i}^{R^{(c)}}(R^{(c,k)})+ \lceil
(s-r)/c\rceil, 0\leq k ,r\leq c-1, \beta_{j s}^{R}(M)\neq 0\}\\
&\leq& t_{i}^{R^{(c)}}(R^{(c, k)})+ \lceil t_j^R(M)/c\rceil.
\end{eqnarray*}
 Now, applying proposition \ref{versyz}, one has
\begin{equation}\label{backineq}
t^{R^{(c)}}_i(F_j)\leq\max\big\{\sum_{v=0}^{u} \lceil
t_{\alpha_{v}}^R(\m(1))/c\rceil+ \lceil t^{R}_{j}(M)/c\rceil: 0\leq
u\leq i, 0\leq \alpha_{v}\leq i, \ \Sigma_{v=0}^{u}\alpha_{v}=i
\big\},
\end{equation}
for all $i,j\geq 0$.

 Set $b:= \max\{\rat(R),\rate_R(M)\}$.
Then, by definition, $t_{\alpha}^R(\m(1))\leq \alpha b$ and
$t^{R}_{\alpha}(M) \leq \alpha b$ for all integer $\alpha>0$. Since
for any real number $x$ and any positive integer $m$ one has $\lceil
mx \rceil\leq m \lceil x \rceil$, we get
$$\lceil t_{\alpha}^R(\m(1))/c\rceil \leq \lceil \alpha b/c\rceil
\leq \alpha \lceil b/c\rceil,$$
and
$$\lceil {t^{R}_{\alpha}(M)}/c\rceil\leq
\lceil \alpha b/c\rceil\leq  \alpha\lceil b/c\rceil,$$ for all
 integer $\alpha>0$. Therefore, in view of the inequality (\ref{backineq}),
we get

\begin{equation}\label{tineq2}
t^{R^{(c)}}_i(F_j)\leq
\begin{cases} (i+j)\lceil b/c\rceil&\text{if}\ \ j>0\\
i\lceil b/c\rceil +  \lceil t^{R}_{0}(M)/c\rceil &\text{if}\ \ j=0,
\end{cases}
\end{equation}
for all $i,j\geq 0$. This, in conjunction with the inequality
(\ref{tineq}), implies that
$$t^{R^{(c)}}_n(M)/n\leq \lceil b/c\rceil+\max\{ 0, \lceil t^{R}_{0}(M)/c\rceil\},$$

for all $n\geq 1$. Hence, we get
$$\rate_{R^{(c)}}(M)\leq \lceil
b/c\rceil+\max\{ 0, \lceil t^{R}_{0}(M)/c\rceil\},$$ as desired.


In the case where $M= \m^s(s)$, for some integer $s\geq 1$, using
lemma \ref{maxi}, we have
$$\rate_{R^{(c)}}(\m^s(s))\leq
\lceil\rat(R)/c\rceil.$$

\end{proof}

Let $R\simeq S/I$, where $S$ is a polynomial ring over $K$ and $I$ a
homogeneous ideal of $S$. Aramova, B$\breve{a}$rc$\breve{a}$nescu
and Herzog in \cite{ABH} showed that for all finitely generated
graded $R$-module $M$ which generated in degree zero,
$$\rate_{R^{(c)}}(M)\leq \max\{ \lceil \rate_{R}(M)/c \rceil,1\},$$
for all $c\geq \max\{1,\rate_S(R)\}.$
 Moreover, by their result,
 if $R$ is a polynomial ring then, the inequality holds for all $c\geq1$.
 Using (\ref{ratineq}) in remark \ref{ratrem}, it is straightforward
 to  see that
 $$\rat(R)\leq \max\{\rate_S(R),1\}.$$
 As an immediate consequence of the above theorem, we have
 the following corollary  that improves the theorem of Aramova et al.

\begin{cor}
Let $M$ be a finitely generated graded $R$-module generated in
degrees zero. Then
$$\rate_{R^{(c)}}(M)\leq \max\{ \lceil
\rate_{R}(M)/c \rceil,1\}, $$
 for all $c\geq \rat(R)$.
 In particular, in the case where $R$ is a Koszul algebra the above
inequality holds for all $c\geq 1$.
\end{cor}
Backelin (\cite{B}) used complex arguments about a lattice of
ideals, derived from a presentation of $R$ as a quotient of a free
noncommutative algebra, to prove that the $c$-th Veronese subring
$R^{(c)}$ is a Koszul algebra for all sufficiently large values of
$c$. Indeed he showed that $\rat(R^{(c)})\leq \lceil\rat(R)/c
\rceil$. The next corollary presents a simple proof for the theorem
of Backelin.
\begin{cor}
Let $R$ be a standard graded $K$-algebra. Then
$$\rat(R^{(c)})\leq \lceil \rat(R)/c\rceil.$$
\end{cor}
\begin{proof}
Let $\m$ be the homogeneous maximal ideal of $R$. Then $\m^{(c)}$ is
the homogeneous maximal ideal of $R^{(c)}$ and it is a direct
summand of $\m$ as an $R^{(c)}$-module.
 Hence, by Theorem \ref{mainthm}, we get
$$\rat(R^{(c)})= \rate_{R^{(c)}}(\m^{(c)}(1))\leq \rate_{R^{(c)}}(\m(1))
\leq \lceil \rat(R)/c\rceil.$$
\end{proof}

Conca (\cite{AC}) showed that if $R$ is Koszul then,
$\reg_{R^{(c)}}(R^{(c,d)})=0$ for all integers $c, d$ with $0\leq
d\leq c-1$. The third part of the following corollary, also,
generalize this result.

\begin{cor}
Let the situations be as in the above theorem. Then the followings
hold.

\begin{enumerate}
\item  If $M$ is generated  in degree zero, then for all $c\geq \max \{
\rate_{R}(M),\rat(R)\},$
$$\reg_{R^{(c)}}(M)=0.$$
\item For all $c\geq \rat(R)$ and $s\geq 1$,
$$\reg_{R^{(c)}}(\m^s(s))=0.$$
\item For all $c\geq \rat(R)$
$$\reg_{R^{(c)}}(R)=0.$$
In particular, $\reg_{R^{(c)}}(R^{(c,d)})=0$ for all $c\geq \rat(R)$
and $0\leq d\leq c-1$.
\end{enumerate}
\end{cor}
\begin{proof}
One can prove the claims, using theorem \ref{mainthm} and noting
that for a finitely generated graded $R^{(c)}$-module $N$ generated
in degree zero, $\reg_{R^{(c)}}(N)=0$ if and only if
$\rate_{R^{(c)}}(N)=1$.
\end{proof}

 \vskip 1 cm



\end{document}